\documentclass[11pt]{amsart}
\usepackage{xparse,amssymb,amsmath,amsthm,mathrsfs}

\usepackage{tabularx,array}
\usepackage[top=1in, bottom=1.25in, left=1.25in, right=1.25in]{geometry}
\usepackage{amssymb,amsmath,amsthm,mathrsfs}
\usepackage{enumitem}
\usepackage{tikz-cd}
\usetikzlibrary{cd}

\usepackage{url}
\usepackage{cite}
\usepackage{amsfonts}
\usepackage{verbatim}
\usepackage{graphicx}
\usepackage{caption} 
\usepackage[all]{xy}
\usepackage{colonequals}

\theoremstyle{plain}
\newtheorem{thm}{Theorem}[section]
\newtheorem*{thm*}{Theorem}

\newtheorem{prop}[thm]{Proposition}

\newtheorem{remark}[thm]{Remark}

\def\R{{\mathbb{R}}}

\def\SL{{\text{\rm{SL}}}}

\def\Sp{{\text{\rm{Sp}}}}

\def\Ho{H_1^{(0)}}

\usepackage{color}
\definecolor{purple}{rgb}{0.5,0,1}
\definecolor{darkgreen}{rgb}{0.1,0.4,0.2}
\definecolor{darkyellow}{rgb}{0.6,0.6,0.2}

\begin{document}
\title[Arithmetic Monodromy group]{An arithmetic Kontsevich--Zorich monodromy of a symmetric origami in genus 4}

\author{Xun Gong}
\address{Department of Mathematics, University of California San Diego\\ 9500 Gilman Dr, La Jolla,
CA 92093, USA}
\email{x1gong@ucsd.edu}

\author{Anthony Sanchez}
\address{Department of Mathematics, University of California San Diego\\ 9500 Gilman Dr, La Jolla,
CA 92093, USA}
\email{ans032@ucsd.edu}


\maketitle
\begin{abstract}
    We demonstrate the existence of a certain genus four origami whose Kontsevich--Zorich monodromy is arithmetic in the sense of Sarnak.  
    The surface is interesting because its Veech group is as large as possible and given by $\mathrm{SL}(2,\mathbb Z)$. When compared to other surfaces with Veech group $\mathrm{SL}(2,\mathbb Z)$ such as the Eierlegendre Wollmichsau and the Ornithorynque, an arithmetic Kontsevich--Zorich monodromy is surprising and indicates that there is little relationship between the Veech group and monodromy group of origamis. Additionally, we record the index and congruence level in the ambient symplectic group which gives data on what can appear in genus 4.
\end{abstract}

\section{Introduction}
The Kontsevich–Zorich monodromy group encodes the homological data of translation surfaces along $\textrm{SL}(2, \mathbb R)$-orbits. Due to the intersection form on the underlying surface, the Kontsevich–Zorich monodromy group lies as a subgroup of a symplectic group and creates a natural bridge between the geometric objects of translation surfaces and the Kontsevich–Zorich monodromy group and its algebraic nature.
One such question that has received much investigation is whether the Kontsevich–Zorich monodromy group is arithmetic or is thin in the sense of Sarnak \cite{PS}. An origami has \emph{arithmetic Kontsevich–Zorich monodromy} if the Kontsevich–Zorich monodromy group is Zariski dense and finite index in the ambient symplectic group and \emph{thin} if the Kontsevich–Zorich monodromy group is Zariski dense and infinite index.

In genus 2, the arithmeticity of all origami was established by  M\"{o}ller, see also Appendix B of \cite{long} for a proof.  In genus 3, an origami with arithmetic monodromy in the minimal stratum was established in Hubert--Matheus \cite{MR4120783}. Shortly thereafter, an infinite family of origami in genera 3, 4, 5, and 6 with arithmetic monodromy were exhibited in \cite{long,2301.06894}. 

Assuming arithmeticity has been established, finer questions regarding the algebraic nature of the monodromy group, such as the index and what type of group is obtained, have been considered. Experimental data \cite{long} shows that for genus 2 origamis with two conical singularities and less than 22 squares, the index is 1, 3, 4, 6, 12, or 24. For genus 2 origami with a single conical singularity, Kattler \cite{Kattler} has shown the index is always at most 3. Concerning what type of groups are obtained, in the context of symplectic groups, there is a positive solution to the congruence subgroup property. That is, all finite index subgroups of symplectic groups, such as the Kontsevich–Zorich monodromy groups in the works mentioned above, are a congruence group of some level $l$. In \cite{long}, they show a specific origami in genus 3 is conjugate, up to finite index, to a congruence subgroup of level 16 and of index 46080 in $\textrm{Sp}(4,\mathbb Z)$. 

Our main result is a new addition to the list of origami with arithmetic monodromy. Later in the introduction, we will see that this surface being arithmetic is surprising and relate this to the monodromy of other surfaces. Additionally, we record the first index and congruence level of a surface in genus 4 which yields experimental data for what can occur in the context of genus 4 origami. See Figure \ref{fig:surface} for an image of the surface.

\begin{figure}
    \centering
    \includegraphics[scale=0.58]{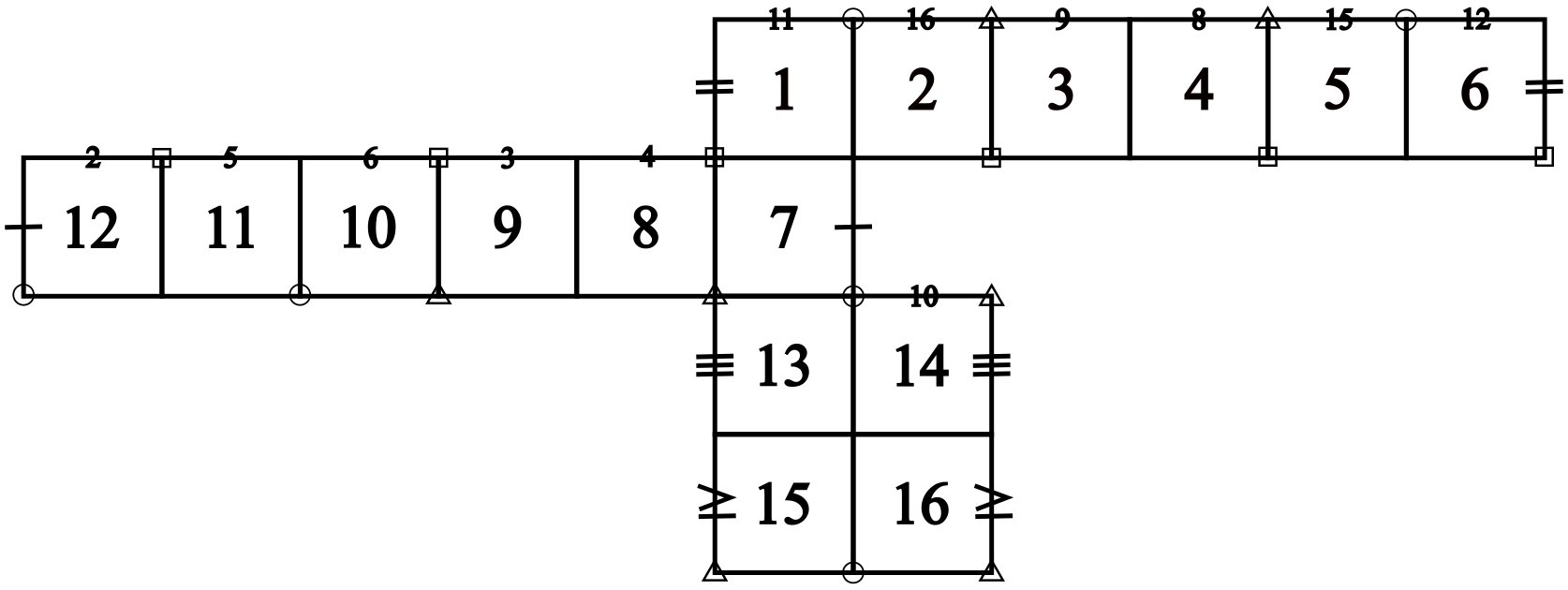}
    \caption{The origami $\mathcal O$ with identification labeled by a permutation.}
    \label{fig:surface}
\end{figure}

\begin{thm}\label{main}
The non-tautological part of the Kontsevich–Zorich monodromy group $\Gamma$ associated to a certain origami $\mathcal O$ of genus 4 is arithmetic. More precisely, up to a finite index subgroup that conjugates $\Gamma$ to a subgroup of the standard symplectic group $\Sp(6,\mathbb Z)$, the Kontsevich–Zorich monodromy group is a congruence subgroup of level 64 and of index $4156153952993280 = 2^{42}\cdot3^3\cdot5\cdot7$.
\end{thm}
As we will see in Section \ref{homology}, the intersection matrix of the homology we use does not have determinant one, so we must pass to a finite index subgroup of $\Gamma$ in order to be a subgroup of the standard symplectic group. The index and congruence level are computed through GAP packages developed by Detinko--Flannery--Hulpke \cite{MR3739225} and Kattler--Weitze-Schmith\"{u}sen \cite{code}. Theorem \ref{main} results in the first data point of the possible indices and congruence levels of origami in genus 4.

\textbf{Relation to other results.} While an infinite family of arithmetic monodromies in genus 4 was established by Kany--Matheus \cite{2301.06894}, we believe the results of this article are worthwhile because of the unique properties of the surface.  What drew our attention to this origami was that it is ``exceptionally symmetric" in that its Veech group is as large as possible and given by $\mathrm{SL}(2,\mathbb Z).$ Additionally, the homological dimension is always 2. We now give some context as to why these points are interesting.

It was shown experimentally in Shrestha--Wang \cite{MR4516259}, that there are only 4 origamis of less than 17 squares whose Veech group is $\mathrm{SL}(2,\mathbb Z)$: the torus, the \emph{Eierlegendre Wollmichsau}, the \emph{Ornithorynque}, and the surface considered in this paper. We refer to \cite{MR4516259} for the description of the Eierlegendre Wollmichsau and Ornithorynque. These two surfaces have been the focus of much study \cite{MR2729331,MR2387362} because of their peculiar properties e.g. their Veech group is $\mathrm{SL}(2,\mathbb Z)$ and they have totally degenerate Lyapunov spectrum. What's most relevant for the discussion here is that Matheus--Yoccoz \cite{MR2729331} showed the Kontsevich–Zorich monodromy groups of these two surfaces are finite groups. That the surface we consider in this article is arithmetic and thus ``large"  is quite interesting given that the monodromies of the Eierlegendre Wollmichsau and Ornithorynque are at the other end of the spectrum. This indicates that there is little relationship between the symmetries of the surface (the Veech group) and the Kontsevich–Zorich monodromy group of the surface. 

Another interesting aspect of our surface is that the homological dimension is always two. Surfaces with directions of homological dimension one are never thin since Dehn multitwists in those directions produce elements in the kernel of the action on the homology (see Section 4.3 of \cite{MR4120783}). Thus, our surface was a candidate to produce an origami with thin monodromy. 

\textbf{Organization.} As in \cite{MR4120783}, we assume some familiarity with the basic features of origamis including their representation theory, homology, and Veech and affine groups of origamis. We recommend the survey of Matheus \cite{MR4563316}. The rest of the note is dedicated to proving Theorem \ref{main}. The organization of the paper is as follows. In Section \ref{Origami}, we describe some properties of the origami we are studying. In Section \ref{homology}, we find a basis for the homology of our origami. In Section \ref{KZmonodromy}, we compute generators for the monodromy group. Lastly, in Section \ref{arithmetic}, we utilize an arithmeticity criterion of \cite{MR3165424} that was explicated in \cite{long} in the context of monodromies of origamis to prove Theorem \ref{main}. 

\subsection{Acknowledgements} The authors would like to thank Carlos Matheus Silva Santos and Pascal Kattler for a careful reading of the manuscript. In particular, we thank Matheus for comments on the Lyapunov spectrum and the reference \cite{MR2218004} and Pascal Kattler for pointing out a mistake in a previous computation of $\alpha(S)$. Lastly, we are immensely grateful to Pascal Kattler and Gabriela Weitze-Schmith\"{u}sen for running the generators of Theorem \ref{thm:generators} through their software \cite{code}. A.S. was supported by the National Science Foundation Postdoctoral Fellowship under grant number DMS-2103136.
\section{The origami $\mathcal O$}\label{Origami}
Let $\mathcal O$ denote the origami (see Figure \ref{fig:surface} for an image of the surface) defined by the permutations
$$h=(1,2,3,4,5,6)(12,11,10,9,8,7)(13,14)(15,16)$$ 
and $$v=(12,2,16,14,10,6)(11,5,15,13,7,1)(3,9)(4,8).$$
The commutator $$[h,v]:=vhv^{-1}h^{-1}=(1,3,5)(2)(4)(6)(7,15,9)(8)(10,16,12)(11)(13)(14)$$ has 3 non-trivial cycles of length 3 and so $\mathcal O\in \mathcal{H}(2,2,2)$ and has genus 4.

As mentioned in the introduction, one aspect that makes this surface so special is its Veech group.  The Veech group is the stabilizer of $\mathcal O$ with respect to the $\textrm{SL}(2,\mathbb R)$ action on the space of translation surfaces. let $T=
\begin{pmatrix}
1&1\\
0&1
\end{pmatrix}$ and $S = 
\begin{pmatrix}
    1&0\\
    1&1
\end{pmatrix}$ and recall that $\mathrm{SL}(2,\mathbb Z)$ is generated by these two matrices.
\begin{prop}
    The Veech group of $\mathcal O$ is $\mathrm{SL}(2,\mathbb Z)$.
\end{prop}
\begin{proof}
At the level of permutations, the generators of $\mathrm{SL}(2,\mathbb Z)$ act by the rules $T(h,v)=(h,vh^{-1})$ and $S(h,v)=(hv^{-1},v)$. Since we care about the surfaces themselves rather than the labeling of the squares, the pair $(h,v)$ is defined up to simultaneous conjugations i.e. $(h,v)$ and $(\psi h \psi^{-1},\psi v \psi^{-1})$ are regarded as the same origami. With this in mind, we let 
$$\psi=(1,3,5)(2,4,6)(7,10)(8,11)(9,12)(13)(14)(15,16)$$
and $$\phi=(1,15)(2,14,6)(3,9)(4)(5,7)(8)(10,12,16)(11,13).$$ A computation shows that $\psi 
\circ T(h,v) \circ \psi^{-1}=(h,v)$ and $\phi \circ S(h,v) \circ \phi^{-1}=(h,v)$. 
\end{proof}

    \begin{remark}
    We remark that the work of Herrlich \cite{MR2218004} yields that any origami can be covered by one whose Veech group is $\SL(2,\mathbb Z)$. However, explicit computations outside of the present article seem difficult because the genus of the cover can grow quickly. For example, for the L-shaped origami of genus 2 and with 3 squares, the corresponding cover is of genus 37 and has 108 squares.
    \end{remark}

In the introduction, we also claimed that the homological dimension is always two. That is, in any rational direction, the surface $\mathcal O$ decomposes into cylinders whose waist curves span a two-dimensional subspace of the homology. Such surfaces have the potential for a thin monodromy group.

\begin{prop}\label{homologicaldimension}
    The homological dimension of $\mathcal O$ in any rational direction is 2.
\end{prop}
\begin{proof}
The surface $\mathcal O$ has two horizontal cylinders of length six and two horizontal cylinders of length two. Since the waist curves on the horizontal cylinders of length six disconnect $\mathcal O$, then they are homotopic. Similarly, the waist curves on the horizontal cylinders of length two are homotopic. By noting that the long and short waist curves are not homotopic, then we deduce that the homological dimension in the horizontal direction is two. Lastly, since the Veech group of $\mathcal O$ is $\mathrm{SL}(2,\mathbb Z)$ and any other cylinder direction can be moved to the horizontal, then we conclude the homological direction in any rational direction is two. 
\end{proof}

Additionally, the homological dimension of $\mathcal O$ has consequences on the Lyapunov spectrum of the Kontsevich-Zorich cocycle over $\textrm{SL}(2,\mathbb R)\cdot \mathcal O$. By Forni's theorem \cite{MR2820565}, the second Lyapunov exponent is positive. 
\section{Homology of $\mathcal O$}\label{homology}
In this section, we give a basis for the homology of $\mathcal O$ and define certain subspaces needed for later sections. We will work with the following basis of the absolute homology group $H_1(\mathcal O,\mathbb R)$ pictured in Figure \ref{fig:basis} that we denote by $\{\gamma_1,\ldots,\gamma_8\}$. The colors, orientations, and slopes are indicated in Table \ref{table:basis}.
\begin{figure}
    \centering
    \includegraphics[scale=0.58]{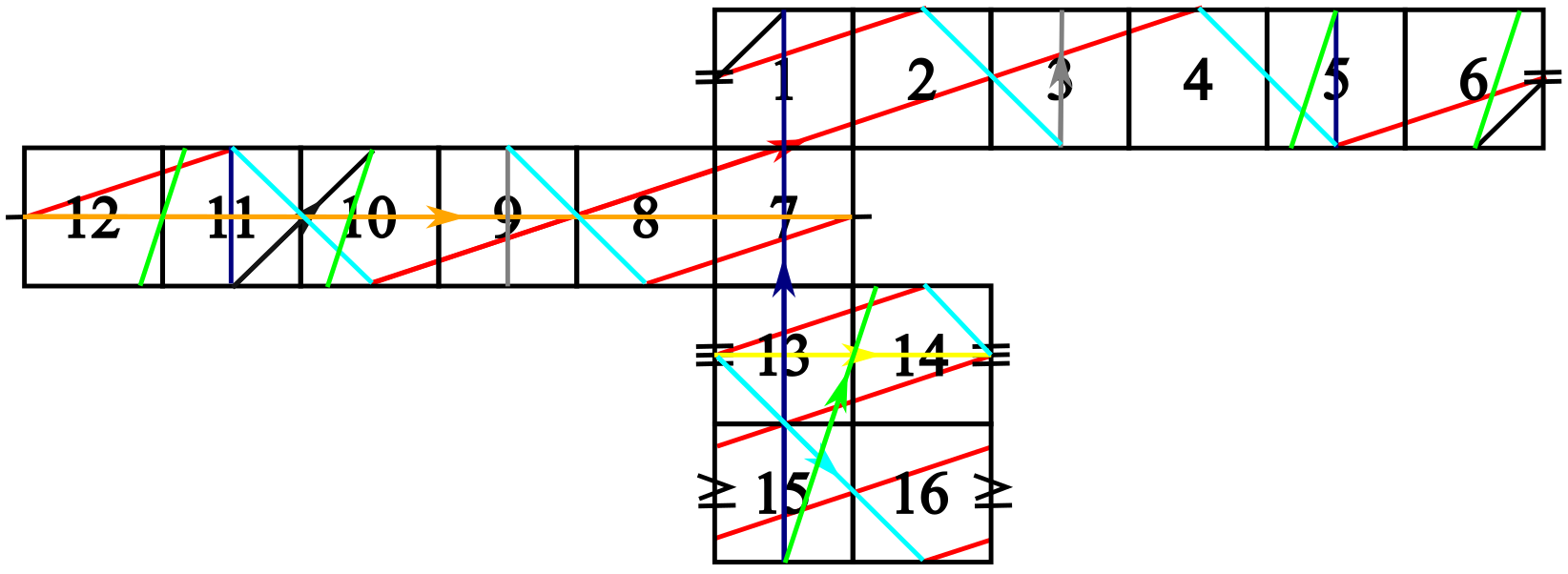}
    \caption{Basis for the homology of $\mathcal O$}
    \label{fig:basis}
\end{figure}

\begin{table}[h!]
    \centering
\begin{tabular}{ |c||c|c| c| c| } 
 \hline
  curve & color & slope & holonomy\\ \hline\hline
 $\gamma_1$ & yellow & 0 & $(2,0)^T$ \\  \hline
 $\gamma_2$ & black &  1  & $(2,2)^T$ \\  \hline
 $\gamma_3$ & orange & 0 & $(6,0)^T$\\  \hline
 $\gamma_4$ & red & 1/3 & $(18,6)^T$\\  \hline
 $\gamma_5$ & grey & $\infty$ & $(0,2)^T$ \\  \hline
 $\gamma_6$ & light blue & -1 & $(6,-6)^T$\\  \hline
 $\gamma_7$ & blue & $\infty$ & $(0,6)^T$\\  \hline
 $\gamma_8$ & green & 3  & $(2,6)^T$\\  \hline

\end{tabular}
\caption{Some properties of a basis of the absolute homology of $\mathcal O$.}
\label{table:basis}
\end{table}

The corresponding intersection matrix of our basis of homology is:
$$\Omega=\bordermatrix{&\gamma_1&\gamma_2&\gamma_3&\gamma_4&\gamma_5&\gamma_6&\gamma_7&\gamma_8\cr
                \gamma_1&0&0&0&1&0&-1&1&1\cr
                \gamma_2&0&0&-1&-1&0&-1&1&1\cr
                \gamma_3&0&1&0&2&1&-2&2&2\cr
                \gamma_4&-1&1&-2&0&2&-9&7&5\cr
                \gamma_5&0&0&-1&-2&0&-1&0&0\cr
                \gamma_6&1&1&2&9&1&0&2&3\cr
                \gamma_7&-1&-1&-2&-7&0&-2&0&-1\cr
                \gamma_8&-1&-1&-2&-5&0&-3&1&0\cr}.$$
Each entry of the matrix, $\Omega_{i,j}$, represents the intersection form $\langle\gamma_i,\gamma_j\rangle$, $i,j\in \{1,2,\dots,8\}$. The determinant of $\Omega$ is 16.

There is always the decomposition of the absolute homology 
$H_1(\mathcal O)=H_1 ^{\mathrm{st}}(\mathcal O)\oplus H_1 ^{(0)}(\mathcal O)$
where $H_1 ^{\mathrm{st}}(\mathcal O)$ the \emph{tautological plane} spanned by the
real and imaginary parts of the implicit Abelian differential and $H_1 ^{(0)}(\mathcal O)$ is the
$6$-dimensional orthogonal complement (with respect to the intersection form) given by the \emph{zero-holonomy subspace}. The name ``zero-holonomy" comes from the fact that the horizontal and vertical displacement of the curves are zero. Next section we justify the terminology of ``tautological plane".
In the setting of our specific surface, it can be checked that with respect to the basis $\{\gamma_1,\ldots,\gamma_8\}$ we have
$$H_1^{st}(\mathcal O) =\{2(\gamma_1+\gamma_3),2(\gamma_5+\gamma_7)\} \text{ and } 
H_1 ^{(0)}(\mathcal O) = \{\epsilon_1,     \epsilon_2,    \epsilon_3,   \epsilon_4,  \epsilon_5,    \epsilon_6\}$$
where
\begin{align*}
    \epsilon_1&=-3\gamma_1+\gamma_3,& \epsilon_2&=-6\gamma_1-3\gamma_2+\gamma_4,\\
    \epsilon_3&=\gamma_1-\gamma_2+\gamma_5,& \epsilon_4&=-6\gamma_1+3\gamma_2+\gamma_6,\\
    \epsilon_5&=3\gamma_1-3\gamma_2+\gamma_7,&\epsilon_6&=2\gamma_1-3\gamma_2+\gamma_8.
\end{align*}


\section{Kontsevich--Zorich monodromy group of $\mathcal O$}\label{KZmonodromy}

In this section, we compute the Kontsevich--Zorich monodromy group of $\mathcal O$. We define these terms and encourage the reader to see Matheus \cite{MR4563316} for more details. Let $\tilde\alpha:\mathrm{Aff}(\mathcal O)\to \mathrm{Sp}(H_1(\mathcal O,\mathbb R))$  denote the representation arising from the action of the affine diffeomorphisms $\mathrm{Aff}(\mathcal O)$ on the absolute homology group $H_1(\mathcal O,\mathbb R)$. By noting that the automorphism group of $\mathcal O$ is trivial, we can (and do) identify $\mathrm{Aff}(\mathcal O)$ with the Veech group of $\mathcal O$.

The representation $\tilde \alpha$ arising from the homological action of $\mathrm{Aff}(\mathcal O)$ respects this decomposition.
In fact, the name ``tautological" comes from the fact that $\Tilde{\alpha}(g)\mid_{H_1^{st}}=g$. Since the action of $\mathrm{Aff}(\mathcal O)$ is well understood on the tautological plane, we restrict our analysis of the action of $\mathrm{Aff}(\mathcal O)$ to the zero-holonomy subspace. Let $\alpha:\mathrm{Aff}(\mathcal O)\to \textrm{Sp}(H_1^{(0)}(\mathcal O))$ denote this restriction. The  \emph{Kontsevich--Zorich monodromy group} of $\mathcal O$ is defined to be $\alpha(\mathrm{Aff}(\mathcal O))$. 

\begin{thm}\label{thm:generators}
        The Kontsevich-Zorich monodromy group of $\mathcal O$ is generated by the following two matrices:
$$\alpha(T)=\begin{pmatrix}
        1&11&2&-6&7&5\\
        0&-4&-1&3&-3&-2\\
        0&-2&-1&3&-2&-3\\
        0&-3&-1&3&-3&-3\\
        0&2&0&-1&1&1\\
        0&0&0&0&0&-1
    \end{pmatrix}\text{ and } \alpha(S)=\begin{pmatrix}
        1&1&-1&4&-3&-3\\
        0&-1&0&0&0&0\\
        1&6&2&-3&3&3\\
        0&3&1&-3&3&3\\
        1&0&-1&3&-2&-1\\
        0&4&1&-3&3&2
    \end{pmatrix}.$$

\end{thm}

\begin{proof}
Since $\mathrm{Aff}(\mathcal O)=\mathrm{SL}(2,\mathbb Z)=\langle T,S\rangle$, then it suffices to compute the image $\alpha(T)$ and $\alpha(S)$. To this end, we utilize the intersection form on the surface $\mathcal O$. See for example, the proof on page 12 of \cite{GLS}. We first compute the homological action of $\mathrm{SL}(2,\mathbb Z)$ given by the representation $\tilde\alpha$ and then restrict it to $\alpha=\tilde\alpha|_{H^{(0)} _1}$.

We only compute $\tilde\alpha(T)$ and a similar argument works for $\tilde\alpha(S)$. Given a basis curve of the absolute homology of $\mathcal O$, we consider the intersection pattern of this curve after the application of $\tilde\alpha(T)$ against the basis $\{\gamma_1,\ldots,\gamma_8\}$ of $H_1(\mathcal O)$. By expressing $\tilde\alpha(T)$ as a linear combination of the basis and utilizing the intersection form, we create a linear system of equations that we can solve.

\begin{figure}
    \centering
    \includegraphics[scale=0.7]{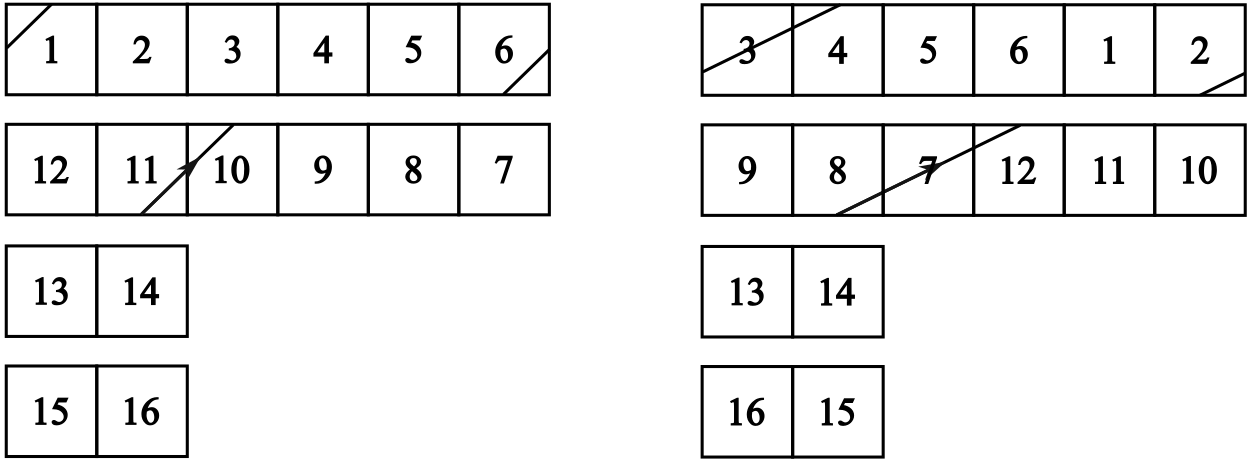}
    \caption{$\gamma_2$ (left) and $\tilde\alpha(T)\gamma_2$ (right)}
    \label{fig:imagegamma2}
\end{figure}

For example, we compute $\tilde\alpha(T)\gamma_2.$ The image is shown in Figure \ref{fig:imagegamma2}. The intersection pattern of $\langle \tilde\alpha(T)\gamma_2,\gamma_i\rangle$ with the basis  can be recorded with the vector $(0,-1,-1,1,-2,1,0)$. We note that this requires using the relabeling map $\psi$ from Section \ref{Origami}. Now writing $\tilde\alpha(T)\gamma_2 = \sum_{i=1} ^{8}a_i\gamma_i$ where $a_i\in\mathbb R$ and using the intersection matrix $\Omega$ we get
$$\langle \tilde\alpha(T)\gamma_2, \gamma_j \rangle=\left\langle \sum\limits_{i=1}^8 a_i\gamma_i,\gamma_j\right\rangle=\sum\limits_{i=1}^8 a_i\langle\gamma_i,\gamma_j\rangle=\sum\limits_{i=1}^8 a_i\Omega_{i,j}.$$
Setting this equal to the vector $(0,-1,-1,1,-2,1,0)$ and solving the linear system yields $\tilde\alpha(T)\gamma_2=-4\gamma_1-2\gamma_3+\gamma_4+\gamma_5+\gamma_6$.
Doing this for all the basis curves and doing a similar computation for $\tilde\alpha(S)$ shows that the homological is given by:

\begin{tabularx}{\textwidth}{@{}l<{,}@{\ }X@{}}
  $\tilde\alpha(T)\gamma_1=\gamma_1$  & $\tilde\alpha(S)\gamma_1=\gamma_2$ \\
  $\tilde\alpha(T)\gamma_2=-4\gamma_1-2\gamma_3+\gamma_4+\gamma_5+\gamma_6$ & $\alpha(S)\gamma_2=2\gamma_1+\gamma_3-\gamma_5-\gamma_6+\gamma_7-\gamma_8$ \\
  $\tilde\alpha(T)\gamma_3=\gamma_3$   & $\tilde\alpha(S)\gamma_3=\gamma_1-\gamma_2+\gamma_3+\gamma_5+\gamma_7$\\
  $\tilde\alpha(T)\gamma_4=7\gamma_1-\gamma_2+5\gamma_3-\gamma_4+\gamma_5+2\gamma_7$ & $\tilde\alpha(S)\gamma_4=5\gamma_1+4\gamma_3-\gamma_4+3\gamma_5+3\gamma_7+\gamma_8,$\\
  $\tilde\alpha(T)\gamma_5=\gamma_2$ & $\tilde\alpha(S)\gamma_5=\gamma_5$\\
  $\tilde\alpha(T)\gamma_6=-\gamma_7$ & $\tilde\alpha(S)\gamma_6=\gamma_3$\\
  $\tilde\alpha(T)\gamma_7=\gamma_1-\gamma_2+\gamma_3+\gamma_5+\gamma_7$ &  $\tilde\alpha(S)\gamma_7=\gamma_7$ \\
  $\tilde\alpha(T)\gamma_8=-\gamma_1-\gamma_3+\gamma_4+\gamma_7-\gamma_8$ & $\tilde\alpha(S)\gamma_8=\gamma_1+\gamma_2+2\gamma_7-\gamma_8$.
\end{tabularx}

Finally, using that the action on the zero-holonomy subspace is just the restriction of the homological action, $\alpha=\tilde\alpha|_{H^{(0)} _1}$, we obtain the matrices in the statement of the theorem with respect to the basis $\{\epsilon_1,\ldots,\epsilon_6\}$.
\end{proof}




\section{Zariski denseness, arithmeticity, index, and congruence level}\label{arithmetic}
In this section, we verify the Zariski density and arithmeticity of the monodromy group $\Gamma=\alpha(\textrm{Aff}(\mathcal O))=\langle\alpha(T),\alpha(S)\rangle$ as well as record the index and congruence level of the conjugate of a finite index subgroup of $\Gamma$. The Zariski density will be verified using a criterion of Kany-Matheus \cite{2301.06894} (bottom of page 2). The arithmeticity will be done by using a criterion of Singh--Venkataramana \cite{MR3165424} and explicated for origamis in \cite{long} page 10. Their metacode uses 2-cylinder decompositions which do not exist for our surface, so we apply a slight modification, see Remark \ref{modification}. The index and congruence level is computed using the computer software of Detinko--Flannery--Hulpke \cite{MR3739225} and Kattler--Weitze-Schmith\"{u}sen \cite{code}.

\begin{thm}
    The monodromy group $\Gamma$ is Zariski dense in $\Sp(6,\R)$.
\end{thm}
\begin{proof}
   According to Kany-Matheus \cite{2301.06894}, it suffices to find two elements $A$ and $B$ in $\textrm{Sp}(H_1 ^{(0)}(\mathcal O))$ such that:
    \begin{enumerate}
        \item The matrix $A$ is a \emph{Galois-pinching}. That is, a symplectic matrix whose characteristic polynomial is irreducible over $\mathbb Z$, splits over $\mathbb R$, and possesses the largest possible (hyperoctahedral) Galois group among reciprocal polynomials of degree 6;
        \item A non-trivial unipotent element $B$ induced by a Dehn twist such that $(B-\mathrm{Id})(H_1 ^{(0)}(\mathcal O))$ is not a Lagrangian subspace of $H_1 ^{(0)}(\mathcal O)$.
    \end{enumerate}

    Let $A=\alpha(S)\alpha(T)\alpha(S)\alpha(T)^{20}$ and $B=\alpha(T)^6$.
    
    The characteristic polynomial of $A$ is  $f_A(x)=1 - 3 x - 91 x^2 - 262 x^3 - 91 x^4 - 3 x^5 + x^6$. Note that this polynomial is reciprocal: $x^6f_A(x^{-1})=f_A(x)$. Additionally, it splits over $\R$ and is irreducible over $\mathbb Z$. Moreover, the permutation group acts on a set of cardinality 6, the order of the Galois group is precisely $2^4\cdot 3=48$ which matches with the order of the hyperoctahedral group. That gives us condition (1).
    
    Now we turn our attention to $B=\alpha(T)^6$. The subspace $(B-\mathrm{Id})(\Ho(\mathcal O))$ is 1 dimensional. Yet, if $(B-\textrm{Id})(\Ho(\mathcal O))$ is a Lagrangian subspace, then $\dim((B-\textrm{Id})(\Ho(\mathcal O)))=\frac{\dim(\Ho(\mathcal O))}{2}=\frac{6}{2}=3\neq 1$. Therefore, we have satisfied condition (2). 
    
    Together with the criterion, the monodromy $\Gamma$ is Zariski dense in $\Sp(6,\R)$.
\end{proof}

 The Zariski-density provided by the above theorem, allows us to say more about the Lyapunov spectrum. In particular, by the main theorem of Eskin-Matheus, we conclude simplicity of the Lyapunov spectrum. That is, the spectrum is of the form 
$$1>\lambda_2>\lambda_3>\lambda_4>0>-\lambda_4>-\lambda_3>-\lambda_2>-1.$$
Some numerical experiments with the \texttt{surface\_dynamics} package \cite{DFL} indicate that $\lambda_2\sim 0.531505241984150, \lambda_3\sim0.277626450819842,$ and $\lambda_4\sim0.191386026412578$.
\begin{remark}\label{modification}
For arithmeticity, the article \cite{long} used 2-cylinder decompositions of their origami, but for our origami no such decomposition exists. Indeed, their are 4 cylinders in the horizontal direction and any other rational direction can be moved to the horizontal using the action of the Veech group $ \mathrm{SL}(2,\mathbb Z)$. The more relevant data for the metacode of \cite{long} is that the homological dimension is 2, so that the associated Dehn multitwist $C$ will have a specific form. Namely, the image of the restriction of $C - \textrm{Id}$ to the zero holonomy part is a one-dimensional subspace.
\end{remark}

\begin{thm}
    The monodromy group $\Gamma$ is arithmetic.
\end{thm}
\begin{proof}
    Using the metacode of \cite{long}, it suffices to find 3 transvections with respect to 3 rational directions of cylinder decomposition. Recall, by the argument from Proposition \ref{homologicaldimension}, that in any rational direction there are two ``long" cylinders with homologous waist curves and two ``short" cylinders with homologous waist curves. We consider three rational directions indicated by the vectors $(1,0)$, $(0,1)$, and $(1,2)$ and consider linear combinations of the waist curves so that we are in the zero-holonomy subspace. The corresponding zero-holonomy vectors of the rational directions $(1,0)$, $(0,1)$, and $(1,2)$ are $w_1=2\gamma_3-6\gamma_1$, $w_2=2\gamma_7-6\gamma_5$, and $w_3=2\alpha-6\beta$. 

    Recall the basis $\{\epsilon_1,\ldots, \epsilon_6\}$ of zero-holonomy subspace from the end of Section \ref{homology}. With respect to this basis, the curves chosen for the transvection are given by $w_1 = 2\epsilon_1$, $w_2= 2\epsilon_5-6\epsilon_3$, and $w_3=-6\epsilon_1+12\epsilon_3+8\epsilon_4-4\epsilon_5+8\epsilon_6$.

The transvections $C_{w_i}$ in the directions $w_i$ for $i=1,2,3$ act on the zero-holonomy subspace 
are given by 
\begin{align*}
C_{w_1} (X) &= X+2 \langle X,\gamma_3\rangle \gamma_3+6\langle X,\gamma_1\rangle \gamma_1,\\
 C_{w_2} (X) &= X+2 \langle X,\gamma_7\rangle \gamma_7+6\langle X,\gamma_5\rangle \gamma_5,\\
 C_{w_3} (X) &= X+2 \langle X,\alpha \rangle \alpha+6\langle X,\beta\rangle \beta.
\end{align*}

Let $e$ be the annihilator of the subspace $W=\textrm{span}_\mathbb Q \{w_1,w_2,w_3\}$. It can be computed by the formula, 
$$e=-\frac{\langle w_3,w_2\rangle}{\langle w_1,w_2\rangle}w_1-\frac{\langle w_3,w_1\rangle}{\langle w_2,w_1\rangle}w_2+w_3=-w_1+2w_2+w_3.$$

Then the transvections $C_{w_i}$ for $i=1,2,3$ with respect to the ordered basis $\{w_1,w_3,e\}$ are
$$C_{w_1}=\begin{pmatrix}
    1&-12&0\\
    0&1&0\\
    0&0&1
\end{pmatrix}, C_{w_2}=\begin{pmatrix}
    0&-1&0\\1&2&0\\-1&-1&1
\end{pmatrix}, C_{w_3}=\begin{pmatrix}
    1&0&0\\4&1&0\\0&0&1
\end{pmatrix}.$$
We find a non-trivial element in $\textrm{Sp}(W)$, by considering the word
$$C_{w_1}C_{w_2}C_{w_1}^{-1}C_{w_2}^{-1}C_{w_3}^3C_{w_1}^{-1}=\begin{pmatrix}
    1&0&0\\0&1&0\\12&144&1
\end{pmatrix}.$$
The metacode of \cite{long} implies that $\Gamma$ is arithmetic.
\end{proof}

We now explain the process of recording the index and congruence level of $\Gamma$. See also Section 6.2 of \cite{long}. Recall that our intersection matrix $\Omega$ has determinant 16. Thus, we cannot conjugate $\Gamma$ to a subgroup of $\Sp(6,\mathbb Z)$. However, we can pass to a specific finite index subgroup that \emph{is} conjugate to a subgroup of $\Sp(6,\mathbb Z)$. These steps allow us to use the software of Detinko--Flannery--Hulpke \cite{MR3739225} and Kattler--Weitze-Schmith\"{u}sen \cite{code}.  We start this process now by changing our basis $\Omega$ to a more standard form. Let
$$\Theta = \left(
\begin{array}{cccccc}
 0 & -1 & 0 & 2 & 0 & 0 \\
 0 & 0 & 0 & -1 & 0 & 0 \\
 0 & 0 & 1 & 1 & 1 & 0 \\
 0 & 0 & 0 & 0 & 0 & -1 \\
 1 & 0 & 0 & 0 & 1 & 0 \\
 -1 & 4 & 0 & 2 & 0 & 0 \\
\end{array}
\right)\text{ so that }
\Theta^t \Omega \Theta = \left(
\begin{array}{cccccc}
 0 & 0 & 0 & 1 & 0 & 0 \\
 0 & 0 & 0 & 0 & 4 & 0 \\
 0 & 0 & 0 & 0 & 0 & 1 \\
 -1 & 0 & 0 & 0 & 0 & 0 \\
 0 & -4 & 0 & 0 & 0 & 0 \\
 0 & 0 & -1 & 0 & 0 & 0 \\
\end{array}
\right).$$
That is, $\Theta$ is a change of basis into a more standard from with respect to a basis $\{b_1,b_2,b_3,b_4,b_5,b_6\}$. Let $\Gamma ' = \Theta^{-1}\Gamma \Theta$. By passing to the sublattice with basis $\{4b_1,b_2,4b_3,b_4,b_5,b_6\}$ and using the software of Kattler--Weitze-Schmith\"{u}sen \cite{code}, we obtain an index 288 subgroup of $\Gamma'$ that can be conjugated to the standard symplectic group.  Applying the software of Detinko--Flannery--Hulpke \cite{MR3739225} yields the congruence level 64 and the index 4156153952993280 in $\Sp(6,\mathbb Z)$.


\begin{thebibliography}{10}

\bibitem{long}
E.~Bonnafoux, M.~Kany, P.~Kattler, C.~Matheus, R.~Ni{\~n}o, M.~Sedano-Mendoza,
  F.~Valdez, and G.~Weitze-Schmith{\"u}sen.
\newblock Arithmeticity of the {K}ontsevich--{Z}orich monodromies of certain
  families of square-tiled surfaces.
\newblock {\em arXiv preprint arXiv:2206.06595}, 2022.

\bibitem{DFL}
V.~Delecroix, C.~Fougeron, and S.~Lelièvre.
\newblock Surface dynamics - sagemath package, version 0.4.1, 2019.

\bibitem{MR3739225}
A.~Detinko, D.~L. Flannery, and A.~Hulpke.
\newblock Zariski density and computing in arithmetic groups.
\newblock {\em Math. Comp.}, 87(310):967--986, 2018.

\bibitem{MR2820565}
G.~Forni.
\newblock A geometric criterion for the nonuniform hyperbolicity of the
  {K}ontsevich-{Z}orich cocycle.
\newblock {\em J. Mod. Dyn.}, 5(2):355--395, 2011.
\newblock With an appendix by Carlos Matheus.

\bibitem{GLS}
R.~Gutiérrez-Romo, D.~Lee, and A.~Sanchez.
\newblock Kontsevich--zorich monodromy groups of translation covers of some
  platonic solids.
\newblock {\em arXiv preprint arXiv:2208.08460}, 2022.

\bibitem{MR2218004}
F.~Herrlich.
\newblock Teichm\"{u}ller curves defined by characteristic origamis.
\newblock In {\em The geometry of {R}iemann surfaces and abelian varieties},
  volume 397 of {\em Contemp. Math.}, pages 133--144. Amer. Math. Soc.,
  Providence, RI, 2006.

\bibitem{MR2387362}
F.~Herrlich and G.~Schmith\"{u}sen.
\newblock An extraordinary origami curve.
\newblock {\em Math. Nachr.}, 281(2):219--237, 2008.

\bibitem{MR4120783}
P.~Hubert and C.~Matheus~Santos.
\newblock An origami of genus 3 with arithmetic {K}ontsevich-{Z}orich
  monodromy.
\newblock {\em Math. Proc. Cambridge Philos. Soc.}, 169(1):19--30, 2020.

\bibitem{2301.06894}
M.~Kany and C.~Matheus.
\newblock Arithmeticity of the {K}ontsevich--{Z}orich monodromies of certain
  families of square-tiled surfaces {II}.
\newblock 2023.

\bibitem{Kattler}
P.~Kattler.
\newblock Index of the {K}ontsevich--{Z}orich monodromy of origamis in
  $\mathcal {H} (2)$.
\newblock {\em arXiv preprint arXiv:2307.00816}, 2023.

\bibitem{code}
P.~Kattler and G.~Weitze-Schmith\"{u}sen.
\newblock Extension of the origami-package: Actions on the homology.
\newblock {\em unpublished GAP package}.

\bibitem{MR4563316}
C.~Matheus.
\newblock Three lectures on square-tiled surfaces.
\newblock In {\em Teichm\"{u}ller theory and dynamics}, volume~58 of {\em
  Panor. Synth\`eses}, pages 77--99. Soc. Math. France, Paris, [2022]
  \copyright 2022.

\bibitem{MR2729331}
C.~Matheus and J.-C. Yoccoz.
\newblock The action of the affine diffeomorphisms on the relative homology
  group of certain exceptionally symmetric origamis.
\newblock {\em J. Mod. Dyn.}, 4(3):453--486, 2010.

\bibitem{PS}
P.~Sarnak.
\newblock Notes on thin matrix groups.
\newblock In {\em Thin groups and superstrong approximation}, volume~61 of {\em
  Math. Sci. Res. Inst. Publ.}, pages 343--362. Cambridge Univ. Press,
  Cambridge, 2014.

\bibitem{MR4516259}
S.~T. Shrestha and J.~Wang.
\newblock Statistics of square-tiled surfaces: symmetry and short loops.
\newblock {\em Exp. Math.}, 31(4):1314--1331, 2022.

\bibitem{MR3165424}
S.~Singh and T.~N. Venkataramana.
\newblock Arithmeticity of certain symplectic hypergeometric groups.
\newblock {\em Duke Math. J.}, 163(3):591--617, 2014.

\end{thebibliography}
\end{document}